\documentclass[12pt]{elsarticle}
\usepackage[left=1in,right=1in, top=1.2in,bottom=1.2in]{geometry}
\usepackage{times}
\usepackage{amssymb,amsthm,latexsym,amsmath,epsfig,pgf}
\usepackage{placeins}
\usepackage{booktabs}
\usepackage{subfig}  
\usepackage{graphicx}  
\usepackage{placeins}
\usepackage{float}
\usepackage{makecell}
\usepackage{xcolor}

\title{
Multiset Partition Dimension of Graphs
}

\newtheorem{theorem}{Theorem}[section]
\newtheorem*{theorem A}{Theorem A}
\newtheorem*{theorem B}{N\"olker's Theorem}
\newtheorem{lemma}{Lemma}[section]

\newtheorem{observation}{Observation}[section]
\newtheorem{problem}{Problem}

\begin{document}

\begin{frontmatter}



\author[label1,label3]{Azzah Albejani}
\ead{azzahahmado.albejani@uon.edu.au}
\author[label1]{Yuqing Lin}
\ead{yuqing.lin@newcastle.edu.au}
\author[label1]{Joe Ryan}
\ead{joe.ryan@newcastle.edu.au}
\author[label2]{Kiki A. Sugeng}
\ead{kiki@sci.ui.ac.id}

\address[label1]{School of Computer and Information Sciences,
The University of Newcastle, Australia}
\address[label2]{Departments of Mathematics,
University of Indonesia,
Depok - Indonesia}
\address[label3]{Department of Mathematics, Faculty of Science, University of Hafr Al Batin, Saudi Arabia}


\begin{abstract}
In this paper, we introduce the multiset partition dimension of graphs. This parameter extends the classical partition dimension to the multiset setting by considering multiset distances from vertices to parts of a vertex partition. We establish some fundamental properties of this parameter and determine its exact values for several important classes of graphs.

\end{abstract}


\end{frontmatter}
\section{Introduction}

The concept of metric dimension of graphs was formally introduced independently by Slater~\cite{slater1975leaves} and Harary \& Melter~\cite{harary1976metric} in the 1970s. 

Let $G = (V,E)$ be a connected graph. For an ordered subset $W = \{w_1, w_2, \dots, w_k\} \subseteq V$, the \emph{metric representation} of a vertex $v \in V$ with respect to $W$ is defined as
\[
r(v|W) = (d(v,w_1), d(v,w_2), \dots, d(v,w_k)),
\]
where $d(u,v)$ denotes the shortest path distance between vertices $u$ and $v$ in $G$. The set $W$ is called a \emph{resolving set} if every pair of distinct vertices $u,v \in V$ has distinct representations, i.e., $r(u|W) \neq r(v|W)$. A resolving set of minimum cardinality is called a \emph{metric basis}, and its size is called the \emph{metric dimension} of $G$, denoted by $\dim(G)$.

Simanjuntak et al. in 2018~\cite{Rinovia2017multiset} have introduced the the multiset variant of the concept. Let $W \subseteq V$. The \emph{multiset representation} of a vertex $v \in V$ with respect to $W$ is defined as
\[
r_m(v|W) = \{ d(v,w) : w \in W \},
\]
where the distances are considered as a multiset (i.e., order is ignored but multiplicities are preserved). The set $W$ is called a \emph{multiset resolving set} if for every pair of distinct vertices $u,v \in V$, we have $r_m(u|W) \neq r_m(v|W)$. A multiset resolving set of minimum cardinality is called a \emph{multiset basis}, and its size is referred to as the \emph{multiset dimension} of $G$.

\textit{The partition dimension} of a graph was first introduced in \cite{chartrand2000} in 2000. It is a variant of the classical metric dimension in which vertices are distinguished based on their distances to subsets of vertices rather than to individual resolving vertex. Formally, let $G$ be a connected graph and let $\Pi = \{P_1, P_2, \dots, P_t\}$ 
be a partition of its vertex set $V(G)$. The partition $\Pi$ is called a 
\textit{resolving partition} if, for every pair of distinct vertices 
$u, v \in V(G)$, the representation of the vertices
\[
r(u|\Pi) = \big(d(u,P_1), d(u,P_2), \dots, d(u,P_t)\big)
\]
and
\[
r(v|\Pi) = \big(d(v,P_1), d(v,P_2), \dots, d(v,P_t)\big)
\]
are distinct, where 
\[
d(u,P_i) = \min\{d(u,x) : x \in P_i\}.
\]
The minimum cardinality $t$ of such a resolving partition is called the 
\textit{partition dimension} of $G$, and is denoted by $pd(G)$. 
 
In this paper, we extend the concept of partition dimension to the multiset setting. 
That is, we consider vertex partitions of a graph such that each vertex is uniquely 
identified by the \emph{multiset} of its distances to the parts of the partition. 
The formal definition is given below.

Let $G$ be a graph, let $v \in V(G)$, and let 
$\Pi = \{S_1, S_2, \dots, S_t\}$ be a partition of the vertex set $V(G)$. 
For each $i \in \{1,2,\dots,t\}$, define
\[
d(v,S_i) = \min\{d(v,x) : x \in S_i\}.
\]
The \textit{multiset partition representation} of $v$ with respect to $\Pi$, 
denoted by $r_{mp}(v \mid \Pi)$, is defined as the multiset
\[
r_{mp}(v \mid \Pi) = \{ d(v,S_1), d(v,S_2), \dots, d(v,S_t) \}.
\]

The partition $\Pi$ is called a \textit{resolving multiset partition} if for every 
pair of distinct vertices $u, v \in V(G)$,
\[
r_{mp}(u \mid \Pi) \neq r_{mp}(v \mid \Pi).
\]

The minimum cardinality $t$ of a resolving multiset partition of $G$ is called the \textit{multiset partition dimension} of $G$, and is denoted by $\mathrm{mpd}(G)$. 
If no resolving multiset partition exists for $G$, then $G$ is said to have infinite multiset partition dimension.

Next, we present some preliminary results.

\section{Basic Properties of the Multiset Partition Dimension}

It is natural to investigate fundamental lower and upper bounds for this newly introduced parameter, as well as to examine its relationship with the classical notion of partition dimension. First, we have the following.

\begin{theorem}
For a connected graph with at least two vertices, the multiset partition dimension of a graph is at least 4.
\end{theorem}

\begin{proof}
If $mpd(G)$ is 1, then there is only one possible representation for the vertices, i.e. $\{0\}$. If $mpd(G)$ is 2, as the two parts of the graph is connected, the two end vertices of the edge connecting the two parts will have the same representation $\{0,1\}$, thus not valid partition.  

Suppose that the $mpd(G)=3$, and there are three parts $S_1$, $S_2$ and $S_3$ in the partition of the graph $G$. Denote the vertices in $S_1$ that have distance one to $S_2$ by $S_1^2$, and similarly define the sets $S_1^3$, $S_2^1$, $S_2^3$, $S_3^1$ and $S_3^2$ as show in the Fig.\ref{fig:mpd=3}. 

It is clear that the representations of the vertices in these sets all contain ${0}$ and ${1}$. Now, consider a vertex $v \in S_1^3$, and let $t \in S_2$ be a vertex closest to $v$, that is, $d(v, S_2) = d(v,t)$. Without loss of generality, assume that a shortest path between $v$ and $t$ contained entirely in $S_3$.

For the vertex $t$, if $d(t,S_1) = d(t,v)$, then both $v$ and $t$ have the same representation $\{0,1,d(v,t)\}$, which is a contradiction. Therefore, we must have $d(t,S_1) < d(t,v)$.

Now consider a shortest path from $t$ to $S_1$. There are two cases. First, suppose that this shortest path passes through $S_3^2$. In this case, we return to a situation analogous to the initial assumption, with the only difference being that the distance between the corresponding vertices is smaller. In the second case, let $d(t,S_1)=d(t,u)$, where $u \in S_1^2$ and shortest path from $t$ to $u$ passes through vertex $w$ from $S_2^1$, now we know $d(u,S_3) \leq d(u, t)+1$. If $d(u,S_3) = d(u, t)+1$, then $w$ and $t$ both have the representation $\{0,1, d(u,t)\}$, thus we know $d(u,S_3) \leq d(u, t)$. If the shortest path between $u$ and $S_3$ go through $S_2$, then we are back to the initial situation with two vertices having smaller distance than the original assumption. If the shortest path between $u$ and $S_3$ is contained in $S_1$, then following the same argument of the second case, we are still back to the initial situation with two vertices having smaller distance than the original assumption.     

As the distance between two vertices are finite, thus repeating the same argument finitely many times yields a contradiction. Thus, the partition is not a valid partition. Hence, we can conclude that $mpd(G) \geq 4$.

\begin{figure}[h!]
    \centering
    \includegraphics[width=0.6\textwidth]{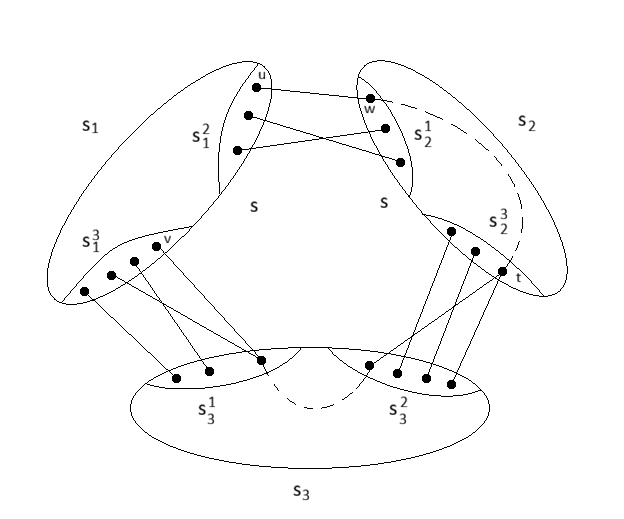}
    \caption{Structure of the graph when mpd=3}
    \label{fig:mpd=3}
\end{figure}

\end{proof}

We can apply the same reasoning as in Lemma 2.2 of \cite{chartrand2000} to the context of the multiset partition dimension, and we have the following. 

\begin{lemma}\label{Lemma2.2partition}
Let $\Pi = \{S_1, S_2, \dots, S_t\}$ be a resolving multiset partition of $V(G)$ and $u,v \in V(G)$. 
If $d(u,w) = d(v,w)$ for all $w \in V(G) \setminus \{u,v\}$, then $u$ and $v$ belong to distinct parts of $\Pi$.   
\end{lemma}

\begin{proof}
Let $\Pi = \{S_1, S_2, \dots, S_t\}$ be the resolving multiset partition, and suppose $u$ and $v$ belong to the same part, say $S_i$ of $\Pi$. 
Then $d(u,S_i) = d(v,S_i) = 0$. 
Since $d(u,w) = d(v,w)$ for all $w \in V(G) \setminus \{u,v\}$, we also have $d(u,S_j) = d(v,S_j)$ for all $j$ ,where $1\le j \neq i \le k$. Therefore, $r_{mp}(u \mid \Pi) = r_{mp}(v \mid \Pi)$ and $\Pi$ is not a resolving partition.
\end{proof}

Following from the Lemma \ref{Lemma2.2partition}, it is easy to see the following.

\begin{observation}
For a tree $T$, if the $mpd(T) \neq \infty$, the leaves from the same parent will belong to different parts of the multiset partition. Furthermore, the parent vertex and one of the leaves must be in the same  part of the partition. 
\end{observation}

\begin{observation}
If $T$ is a tree contains a vertex which is adjacent to at least three leaves, then $mpd(T)=\infty$  
\end{observation}

Apparently, this parameter behaves differently from some related graph parameters. 
For example, in \cite{chartrand2000} Theorem 1.1, it is shown that the partition dimension 
of a graph is at most the metric dimension plus one. However, we are not able to establish a similar result in the multiset setting. 

In \cite{bong2021some}, Novi et al.\ provided a construction of trees that achieve a prescribed value of the multiset metric dimension. At first, we expected that these trees would also exhibit a large multiset partition dimension, however, this turns out not to be the case. The following are examples.

Let $p_i$, $i\ge1$, be the vertices of the spine path, and let $b_i$ be the branch vertex attached to $p_i$ with two leaves $L_i,L_i'$ adjacent to $b_i$. The branches alternate above and below the spine along the path, and the position of the branch vertices $b_i$ on the spine is given by $t_i=1+ \frac{i(i-1)}{2}, \quad i \ge 1.$
The parts are $S_1$ (upper-left), $S_2$ (upper-right), $S_3$ (lower), and $S_4$ (largest, central).

\begin{figure}[h!]
    \centering
    \includegraphics[width=0.8\textwidth]{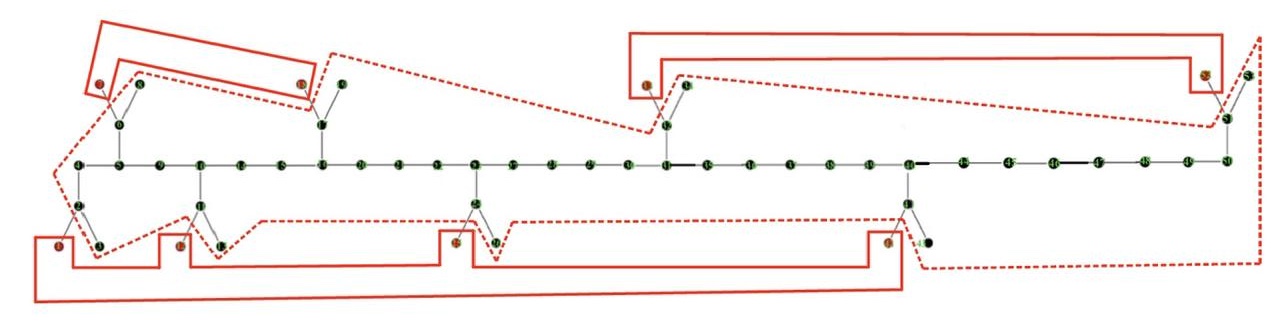}
    \caption{The partitions are 
$S_1=\{L_2,L_{7}\}$, 
$S_2=\{L_{16},L_{29}\}$, 
$S_3=\{L_1,L_{4},L_{11},L_{22}\}$, 
and 
$S_4 = V(G)\setminus(S_1 \cup S_2 \cup S_3)$,
resulting in a total of four parts.
}
    \label{fig:G9withmpd=4}
\end{figure}
\FloatBarrier

\begin{figure}[h!]
    \centering
    \includegraphics[width=1.0\textwidth]{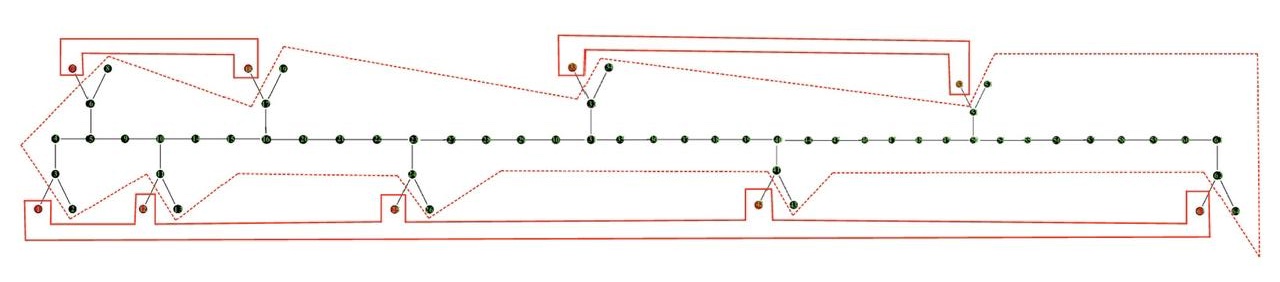}
    \caption{For the graph, the resolving multiset partition 
$\Pi=\{S_1,S_2,S_3,S_4\}$. 
The partitions are 
$S_1=\{L_2,L_{7}\}$, 
$S_2=\{L_{16},L_{29}\}$, 
$S_3=\{L_1,L_{4},L_{11},L_{22},L_{37}\}$, 
and 
$S_4 = V(G)\setminus(S_1 \cup S_2 \cup S_3)$,
resulting in a total of four parts.
}
    \label{fig:G10withmpd=4}
\end{figure}
\FloatBarrier

\begin{figure}[h!]
    \centering
    \includegraphics[width=1.0\textwidth]{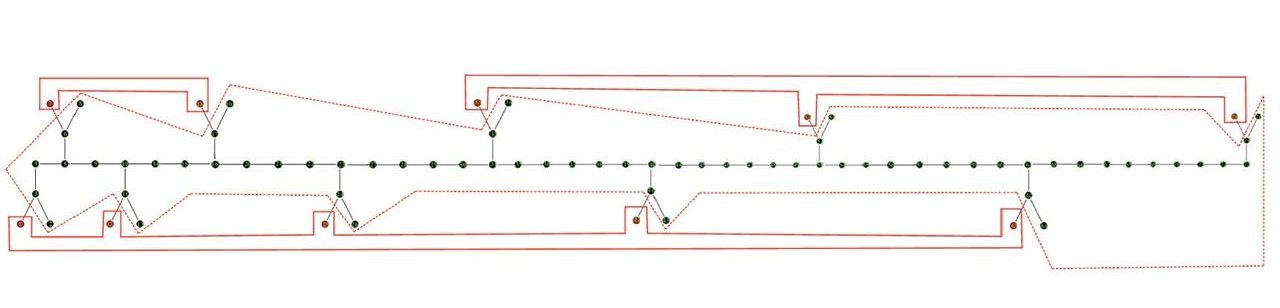}
    \caption{For the graph, the resolving multiset partition 
$\Pi=\{S_1,S_2,S_3,S_4\}$. 
The partitions are 
$S_1=\{L_2,L_{7}\}$, 
$S_2=\{L_{16},L_{29},L_{46}\}$, 
$S_3=\{L_1,L_{4},L_{11},L_{22},L_{37}\}$, 
and 
$S_4 = V(G)\setminus(S_1 \cup S_2 \cup S_3)$,
resulting in a total of four parts.
}
    \label{fig:G11withmpd=4}
\end{figure}
\FloatBarrier

Based on the construction of this tree, when the tree grows, a new vertex $b_j$ is added with two leaves. Since the positions of the two leaves alternate up and down along the path, their assignment to the part depends on their location. If the two added leaves are attached to the upper side of the path, one of them is placed in part $S_2$, if they are attached to the lower side, one of them is placed in part $S_3$. All other newly added vertices are assigned to part $S_4$. 

If $L_i$ and $L_j$ be two leaves of the same part where $i<j$, they must be from different branches. Without loss of generality, we assume that both vertices are from $S_2$, clearly, both leaves have distance $0$ to its own partition $S_2$, and distance $1$ to the partition $S_4$. Their distance to $S_3$ are $d(L_i, L_{t})$ and  $d(L_j, L_{w})$, where $t<i$ and $i<w<j$, based on the construction of the tree, these two distances are different. If both vertices belong to the $S_4$, they might have the same distance to $S_3$, however, it is easy to see they must have different distances to $S_1$ and $S_2$, or they have the same multiset distances to $S_1$ and $S_2$, but then will have different distance to $S_3$, thus will have different representation. The other cases such as two vertices are from different partition can also be easily verified. Thus the multiset partition dimension of the graph is four. 

This observation allows us to conclude the following. 

\begin{theorem}
Let $G$ be a graph with finite multiset metric dimension $md(G)$ and multiset partition dimension $mpd(G)$. Then $|mpd(G) - md(G)|$ can be arbitrarily large.
\end{theorem}

Furthermore, we are able to construct graphs with any prescribed value of the multiset partition dimension, for example, see Figure \ref{fig:MPD=7withk7}.

\begin{figure}[h!]
    \centering
    \includegraphics[width=8cm, height=6cm]{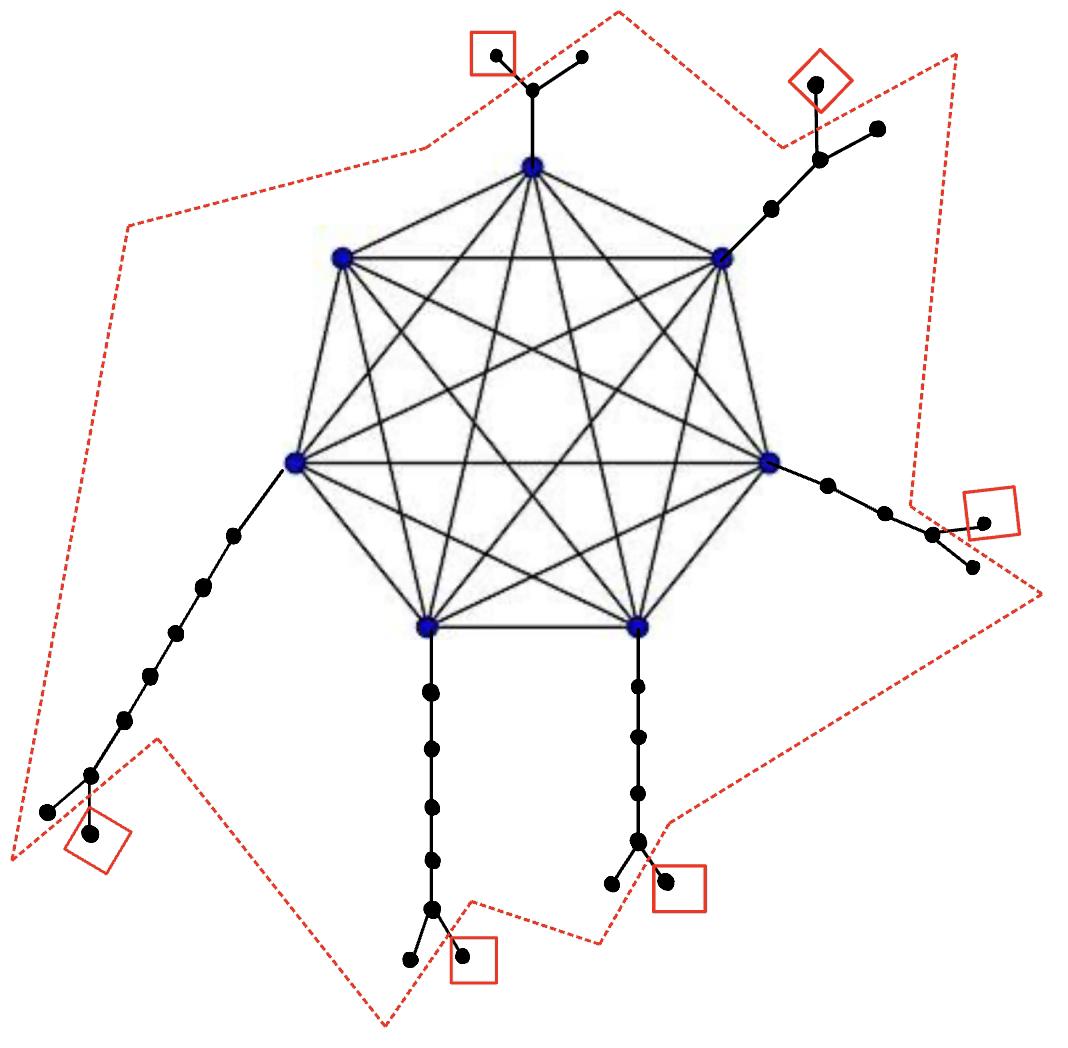}
        \caption{A graph with $\operatorname{mpd}=7$. Parts are highlighted by red boxes}
    \label{fig:MPD=7withk7}
\end{figure}
\FloatBarrier

\begin{theorem}
For any positive integer $k \ge 6$, there exists a graph $G$ whose multiset partition dimension is $k$, that is, $\mathrm{mpd}(G) = k$.
\end{theorem}

\begin{proof}
Let $k \ge 6$. We construct the graph $G_k$ by taking the complete graph $K_{k}$ as the core and attaching to $k-1$ of its vertices a pendant path $L_i$ of length $\ell_i$, where $\ell_i=i$ for $ 1 \le i \le k-1$. At the end of each pendant path two leaves are attached. 
The vertices along the $i$-th path are labeled as

\[
v_{i,0},v_{i,1},\dots v_{i,\ell_i},v_{i,\ell_{i}+1}',v_{i,\ell_{i}+1}''
\]

where $v_{i,0}$ is the core vertex of the complete graph, $v_{i,1},\dots,v_{i,\ell_{i}}$ are the vertices along the path,
and $v_{i,\ell_{i}+1}',v_{i,\ell_{i}+1}''$
are the two leaves attached at the end of the path.

The pair of vertices $v_{i,\ell_{i}+1}'$ and $v_{i,\ell_{i}+1}''$ satisfies the Lemma~\ref{Lemma2.2partition}, so they must both be placed in different partitions and this holds for each of the $k-1$ tails. 

Let $\Pi = \{S_1, S_2, \dots, S_{k-1}, S_k\}$ be a multiset resolving partition where $S_i$ is a singleton partition that contains one of the two leaves $v_{i,\ell_{i}+1}'$ or $v_{i,\ell_{i}+1}''$ for $i=1,\dots,k-1$. Without loss of generality, we take $S_i$ to contain $v_{i,\ell_{i}+1}'$, and $S_k$ to contain all remaining vertices. 

The representations for the leaves are as follows:

\[
r_{mp}(v_{i,\ell_{i}+1}''|\Pi)=
    \{0,2\} \cup \{\ell_i+\ell_j+3 : i \ne j\},\\
\] 

\[
r_{mp}(v_{i,\ell_{i}+1}'|\Pi)=
    \{0,1\} \cup \{\ell_i+\ell_j+3 : i\ne j\}.
\]

Here, $0$ corresponds to the distance from the singleton partition $S_i$ to itself, and $1$ corresponds to its distance to the large partition $S_k$.

The representations for the remaining vertices in the $S_k$ partition are: 
\[
r_{mp}(v_{j,c} \mid \Pi) = \{ d(v_{j,c},S_1), d(v_{j,c},S_2), \dots, d(v_{j,c},S_{k-1}), d(v_{j,c},S_k) \},
\]

where the distance $d(v_{j,c},S_k)=0$ and $d(v_{j,c},S_i)$ are

\[
d(v_{j,c},v_{i,\ell_{i}+1}')=
\begin{cases}
    \left| \ell_i - c +1 \right| ,   i=j\\
    c+2+\ell_i,        i \ne j.
\end{cases}
\]
where $j$ denotes the index of the tail on which the vertex lies, and $c$ denotes the number of steps from the vertex to the nearest core vertex, with $c=0,1,\dots,\ell_j$.
Since all tails have distinct lengths and each core vertex is attached to a different tail, the distance patterns to the singleton partitions are different. Therefore, all vertices have 
unique multiset partition representations.  

If any of the $S_i$s can be combined to form a single part, for example, $v'_{2,3}$ and $v'_{3,4}$ forming a single part, then it is easy to check the vertex $v_{3,0}$ and $v_{k,0}$ will have the same representation, as they have the same multiset distance to all the leaf parts.  

\end{proof}

Note, the similar construction can be used to construct a graph with $mpd(G)=5$, the only modification is to extend the longest tail by one. 

\section{Multiset Partition Dimension of Some Basic Graph Classes}

In this section, we determine the multiset partition dimension of several basic graph classes.

For a path, it is straightforward to see that:
\begin{theorem}
For a path $P_n$ with $n \ge 5$, the multiset partition dimension is 4.
\end{theorem}

\begin{proof}
Let $P_n = v_1, v_2, \ldots, v_n$ and let $\Pi = \{S_1, S_2, S_3, S_4\}$ be the partition of $V(P_n)$ where
\[
S_1 = \{v_1, \ldots, v_{n-4}\}, \quad
S_2 = \{v_{n-3}, v_{n-2}\}, \quad
S_3 = \{v_{n-1}\}, \quad
S_4 = \{v_n\}.
\]

For a vertex $v_i$ in $S_1 = \{v_1, \ldots, v_{n-4}\}$, its representation is:
\[
r_{mp}(v_i \mid \Pi) = 
\{\, 
0,\ 
n-3-i,\ 
n-1-i,\ 
n-i 
\,\}, 
\quad 1 \le i \le n-4.
\]

Where \(0\) corresponds to distance of $v_i$ to the set $S_1$ which containing the vertex itself. The term \(n-3-i\) gives the distance to the nearest vertex in \(S_2\), while the remaining terms are the distances to sets \(S_3\) and \(S_4\). The representation at $i = n-4$ is \[
r_{mp}(v_{n-4} \mid \Pi) = \{0,1,3,4\}.
\] 
For any other vertex $v_i$ in $S_1$, the representation has each nonzero coordinate increased by $n-4-i$, thus all distinct.

The representations for the vertices in $S_2$, $S_3$ and $S_4$ are listed below

\[r_{mp}(v_{n} \mid \Pi) = \{0,1,2,4\}.\]
\[r_{mp}(v_{n-1} \mid \Pi) = \{0,1,1,3\}.\]

\[r_{mp}(v_{n-2} \mid \Pi) = \{0,1,2,2\}.\]

\[r_{mp}(v_{n-3} \mid \Pi) = \{0,1,2,3\}.\]

It is straightforward to see the multiset partition representations for all vertices are different. Hence, $mpd(P_n) = 4$.

\end{proof}

 \begin{figure}[ht]
    \centering
    \includegraphics[width=0.55\textwidth]{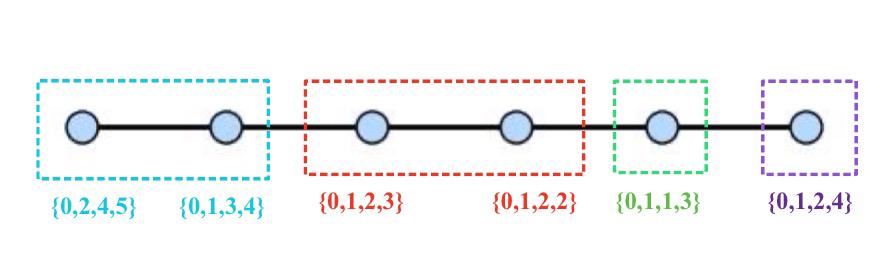}
    \caption{An example illustrating the multiset partition dimension following the patition schema used in the proof for $P_6$.}
    \label{fig:sample}
\end{figure}
\FloatBarrier

For a general path graph $P_n$, see Figure~\ref{fig:GeneralPathwithmpd=4} 

\begin{figure}[h!]
    \centering
\includegraphics[width=1.0\textwidth]{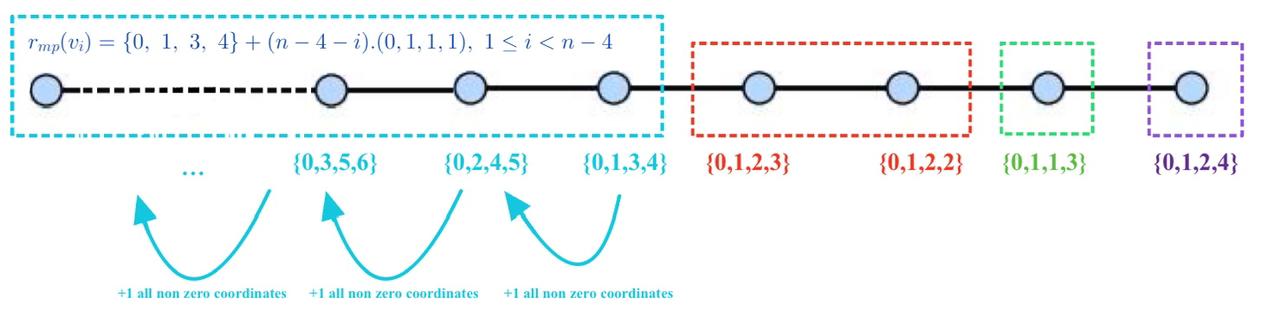}
    \caption{For the path graph $P_{n}$, the partition 
    is illustrated together with the multiset partition representations of vertices, showing that $\operatorname{mpd}(P_{n}) = 4$.}
    \label{fig:GeneralPathwithmpd=4}
\end{figure}
\FloatBarrier

The above partition is not unique, for example, the following is a different partition.

\begin{figure}[ht]
    \centering
    \includegraphics[width=0.55\textwidth]{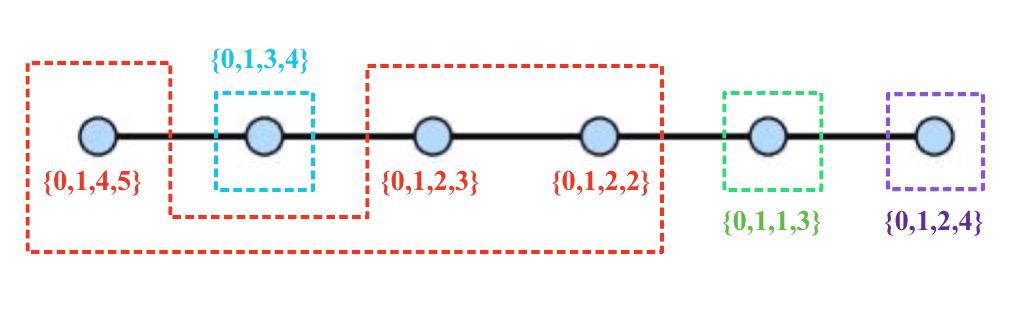}
    \caption{An example of multiset partition dimension}
    \label{fig:sample}
\end{figure}
\FloatBarrier

For a circle, we have the following
\begin{theorem}
 Let $C_n$ be a cycle graph. Then,
\[
\mathrm{mpd}(C_n) =
\begin{cases}
\infty, & \text{if } n \le 7, \\
5, & \text{if } n = 8,\\[2mm]
4, & \text{if } n \ge 9.
\end{cases}
\]
\end{theorem}

\begin{proof}
Let $C_n = v_1, v_2, \ldots, v_n$ and let $\Pi = \{S_1, S_2, S_3, S_4\}$ be the partition of $V(C_n)$ where
\[
S_1 = \{v_1, \ldots, v_{n-5}, v_{n-3}, v_{n-2}\}, \quad
S_2 = \{v_{n-4}\}, \quad
S_3 = \{v_{n-1}\}, \quad
S_4 = \{v_n\}.
\]

For vertices $v_i$ in $S_1 = \{v_1, \ldots, v_{n-5}, v_{n-3}, v_{n-2}\}$, the representation is:
\[
\begin{aligned}
r_{mp}(v_i \mid \Pi)
= \{\, &0,\ \min(|i-(n-4)|,\, n-|i-(n-4)|),\\
      &\min(|i-(n-1)|,\, n-|i-(n-1)|),\\
      &\min(|i-n|,\, n-|i-n|) \,\}
\end{aligned}
\]
for $1 \le i \le n-5$.

Where \(0\) corresponds to the set containing the vertex itself. The second term gives the distance to $S_2$, while the remaining terms are the distances to sets \(S_3\) and \(S_4\).

The representation for $v_{n-3}$ is \[
r_{mp}(v_{n-3} \mid \Pi) = \{0,1,2,3\}.
\] 
The representation at $v_{n-2}$ is \[
r_{mp}(v_{n-2} \mid \Pi) = \{0,1,2,2\}.
\] 

For cycle with diameter $\ge 6$, the representations of the vertices in $S_1$, with $1 \le i \le n-5$, each contain at least one distance $\ge 5$. And it is easy to verify that the representation for vertices in $S_1$ is unique. On the other hand, the representations of the vertices in $S_2$,$S_3$ and $S_4$ contain distances with values at most $4$. Therefore, there can be no overlap between the vertices representations of $S_1$ and those of the other partitions.

For cycles with smaller diameter, such as $C_9$, $C_{10}$, and $C_{11}$, we have verified that there is no overlap in the vertex representations manually. 

The representations for the vertices in $S_2$, $S_3$ and $S_4$ are listed below

\[r_{mp}(v_{n-4} \mid \Pi) = \{0,1,3,4\}.\]
\[r_{mp}(v_{n-1} \mid \Pi) = \{0,1,1,3\}.\]
\[r_{mp}(v_{n} \mid \Pi) = \{0,1,1,4\}.\]

It is straightforward to see the multiset partition representations for all vertices will be different. Hence, $mpd(C_n)= 4 ,\forall n \ge 9$.

For the case when $n <9$, we have verified by computer search.
\end{proof}

\begin{figure}[h!]
    \centering
    \includegraphics[width=0.40\textwidth]{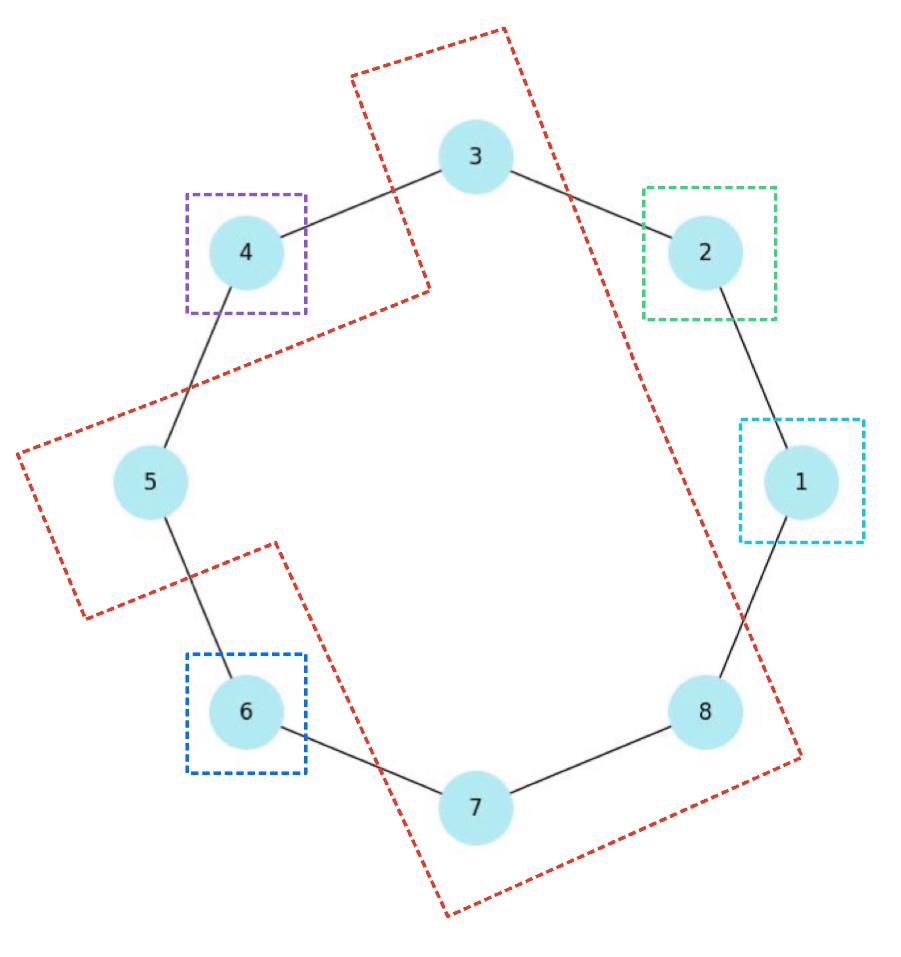}
    \caption{$C_8$ has $\operatorname{mpd}=5$, with partitions illustrated in the figure. Note that $C_8$ cannot be resolved with $\operatorname{mpd}=4$ since the multiset representations of some vertices would contain duplicates.}
    \label{fig:C8withmpd=5}
\end{figure}
\FloatBarrier

\begin{table}[h!]
\centering
\scriptsize
\begin{tabular}{ll}
\toprule
\textbf{Vertex} & 
\textbf{Multiset partition representation}\\
\midrule
1 & $\{0, 1, 1, 3, 3\}$ \\
2 & $\{0, 1, 1, 2, 4\}$ \\
3 & $\{0, 1, 1, 2, 3\}$ \\
4 & $\{0, 1, 2, 2, 3\}$ \\
5 & $\{0, 1, 1, 3, 4\}$ \\
6 & $\{0, 1, 2, 3, 4\}$ \\
7 & $\{0, 1, 2, 3, 3\}$ \\
8 & $\{0, 1, 2, 2, 4\}$ \\
\bottomrule
\end{tabular}
\caption{Multiset partition representations of the vertices for $C_8$.}
\end{table}

\begin{figure}[h!]
    \centering
    \includegraphics[width=0.40\textwidth]{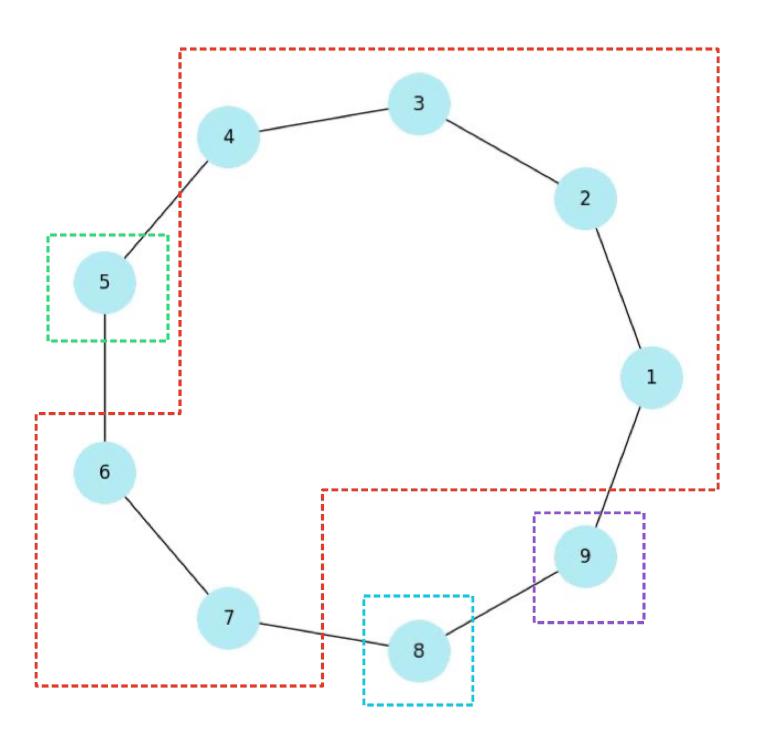}
    \caption{$C_9$ has $\operatorname{mpd}=4$, with partitions illustrated in the figure.}
    \label{fig:C9withmpd=4}
\end{figure}
\FloatBarrier

\begin{table}[h!]
\centering
\scriptsize
\begin{tabular}{ll}
\toprule
\textbf{Vertex} & 
\textbf{Multiset partition representation}\\
\midrule
1 & $\{0, 1, 2, 4\}$ \\
2 & $\{0, 2, 3, 3\}$ \\
3 & $\{0, 2, 3, 4\}$ \\
4 & $\{0, 1, 4, 4\}$ \\
5 & $\{0, 1, 3, 4\}$ \\
6 & $\{0, 1, 2, 3\}$ \\
7 & $\{0, 1, 2, 2\}$ \\
8 & $\{0, 1, 1, 3\}$ \\
9 & $\{0, 1, 1, 4\}$ \\
\bottomrule
\end{tabular}
\caption{Multiset partition representations of the vertices for $C_9$.}
\end{table}

For a general cycle graph $C_n$ with $n\ge 9$, see Figure~\ref{fig:GeneralCyclewithmpd=4} 

\begin{figure}[h!]
    \centering
\includegraphics[width=0.6\textwidth]{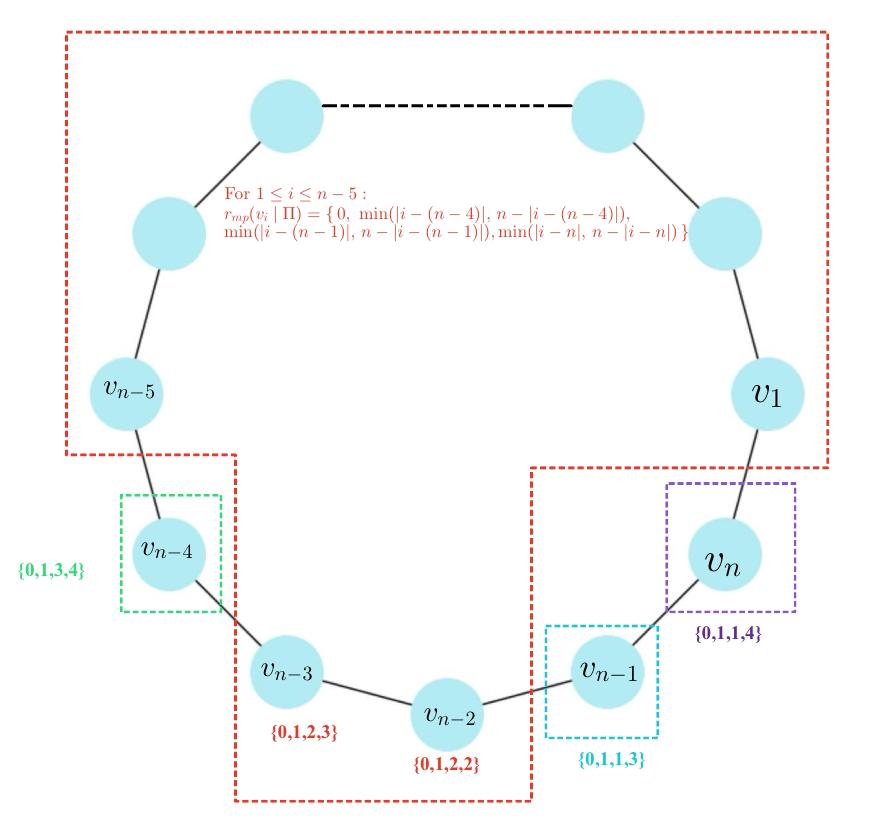}
    \caption{For the cycle graph $C_{n}$ with $n\ge 9$, the partition is illustrated showing that $\operatorname{mpd}(C_{n}),\ n\ge9 = 4$.}
    \label{fig:GeneralCyclewithmpd=4}
\end{figure}
\FloatBarrier

For other graphs, we have..
\begin{theorem}
Let $G_{n,m}$ be a grid graph with $n \ge 3$ and $m \ge 5$. Then
\[
\mathrm{mpd}(G_{n,m})=4.
\]
\end{theorem}

\begin{proof}
Let $G_{n,m}$ be a grid graph with $n \ge 3$ rows and $m \ge 5$ columns.

We label each vertex by an ordered pair 
\[
v=(x,y), \quad x=1,\dots,n, \quad \text{and} \quad  \ y=1,\dots,m. 
\]

Let 
$\Pi=\{S_1, S_2, S_3, S_4\}$
be the partition of $V(G_{n,m})$ where 
\[
\begin{aligned}
S_1 &= \{(1,1)\}, & S_2= \{(1,2)\}, \\
S_3 &= \{(1,m)\}, \\
S_4 &= \{V(G_{n,m}) \setminus (S_1 \cup S_2 \cup S_3)\}. \\
\end{aligned}
\]

For any vertex $v=(x,y) \in V(G_{n,m})$, its multiset partition representation with respect to the above partition is given by 

\[
r_{mp}(v \mid \Pi)=\{d_1,d_2,d_3,d_4\},\] where $d_i$ is the distance from $v$ to $S_i$.

\[
\begin{aligned}
d_1 &= |x-1|+|y-1|,\\
d_2 &= |x-1|+|y-2|,\\
d_3 &= |x-1|+|y-m|,\\
d_4 &= 
\begin{cases}
0, & (x,y)\in S_4,\\
1, & (x,y)\in\{(1,1),(1,2),(1,m)\}.
\end{cases}
\end{aligned}
\]

It can be verified that for any two distinct vertices $u,v \in V(G_{n,m})$, their multiset partition representations are different due to the positions of three sets $S_1, S_2, S_3$ along the grid.
Hence, $\Pi$ is a resolving multiset partition for $G_{n,m}$, and $\mathrm{mpd}(G_{n,m}) = 4.$
\end{proof}
\begin{figure}[h!]
    \centering
    \includegraphics[width=0.6\textwidth]{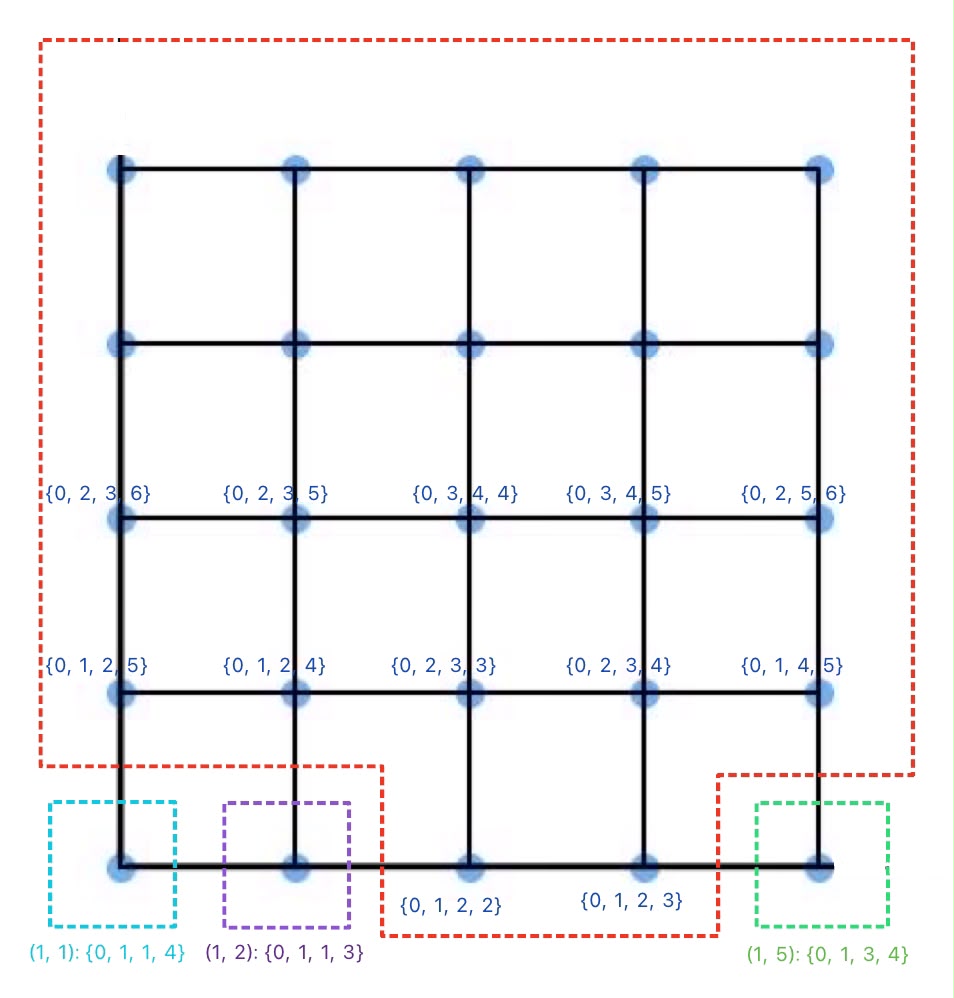}
    \caption{For the grid graph $G_{5,5}$, the partition 
    $\Pi=\{\{(1,1)\}, \{(1,2)\}, \{(1,5)\}, \text{\{all remaining vertices\}}\}$ 
    is illustrated in the figure, and $\operatorname{mpd}(G_{5,5}) = 4$.}
    \label{fig:Gridwithmpd=4}
\end{figure}
\FloatBarrier

\begin{theorem}
Let $L_{m}$ be a ladder graph with $2$  rows and $m \ge 5$ columns. Then
\[
\mathrm{mpd}(L_{m}) = 4.
\]
\end{theorem}

\begin{proof}
Let $L_{m}$ be a ladder graph with $2$ rows and $m \ge 5$ columns. We label each vertex by an ordered pair 
\[
v=(x,y), \quad x=1,2 \quad \text{and} \quad  \ y=1,\dots,m. 
\]

Let 
$\Pi=\{S_1, S_2, S_3, S_4\}$
be the partition of $V(L_{m})$ where 
\[
\begin{aligned}
S_1 &= \{(1,1)\}, & S_2= \{(1,2)\}, \\
S_3 &= \{(1,m)\}, \\
S_4 &= \{V(L_{m}) \setminus (S_1 \cup S_2 \cup S_3\}. \\
\end{aligned}
\]

For any vertex $v=(x,y) \in V(L_{m})$, its multiset partition representation with respect to the above partition is given by 
\[
r_{mp}(v \mid \Pi)=\{d_1,d_2,d_3,d_4\},\] where $d_i$ is the distance from $v$ to $S_i$.

\[
\begin{aligned}
d_1 &=|x-1| +|y-1|, \\
d_2 &=|x-1| +|y-2|, \\
d_3 &= |x-1|+|y-m|,\\
d_4 &= 
\begin{cases}
0, & (x,y)\in S_4,\\
1, & (x,y)\in\{(1,1),(1,2),(1,m)\}.
\end{cases}
\end{aligned}
\]

It can be verified that for any two distinct vertices $u,v \in V(L_{m})$, their multiset partition representations are different due to the positions of three sets $S_1, S_2, S_3$ along the ladder.

Hence, $\Pi$ is a resolving multiset partition for $L_{m}$, and $\mathrm{mpd}(L_{m}) = 4.$
\end{proof}

\begin{figure}[h!]
    \centering
\includegraphics[width=0.55\textwidth]{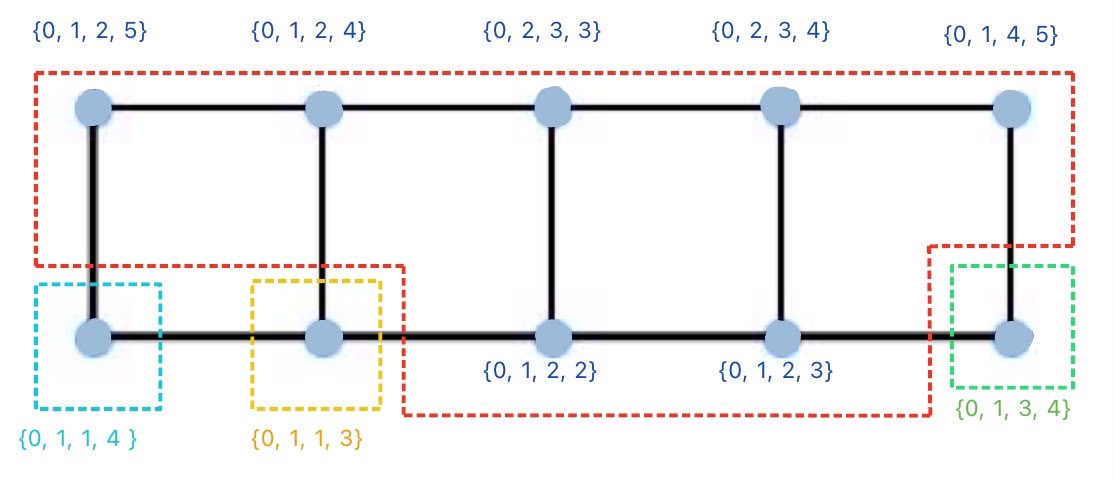}
    \caption{For the ladder graph $L_{5}$, the partition 
    $\Pi=\{\{(1,1)\}, \{(1,2)\}, \{(1,5)\}\},\\ \text{\{all remaining vertices\}}\}$ 
    is illustrated in the figure, and $\operatorname{mpd}(L_{5}) = 4$.}
    \label{fig:Ladderwithmpd=4}
\end{figure}
\FloatBarrier

For a ladder graph $L_m$ with larger $m$, see Figure~\ref{fig:GeneralLadderwithmpd=5} for an example. 

\begin{figure}[h!]
    \centering
\includegraphics[width=0.8\textwidth]{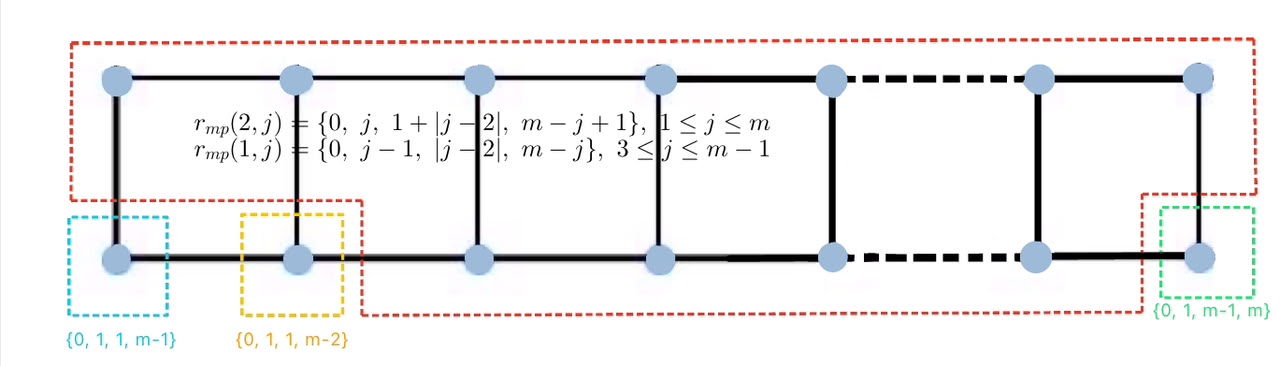}
    \caption{For the ladder graph $L_{m}$, the partition 
    $\Pi=\{\{(1,1)\}, \{(1,2)\}, \{(1,m)\},\text{\{all remaining vertices\}}\}$
    is illustrated together with the multiset partition representations of all vertices.}
    \label{fig:GeneralLadderwithmpd=5}
\end{figure}
\FloatBarrier


For the prism graph $P_m$, we have conducted a computer verification for $m$ up to 70, can confirm that $mpd(P_m)$ is less than or equal to $6$. For $m=6$ and $7$, the partition is illustrated in Fig.~\ref{fig:Prismwithmpd=6} and Fig.~\ref{fig:P7} . And for $m \ge 8$, the partition is discussed in the following theorem.


\begin{figure}[h!]
    \centering
    \includegraphics[width=0.4\textwidth]{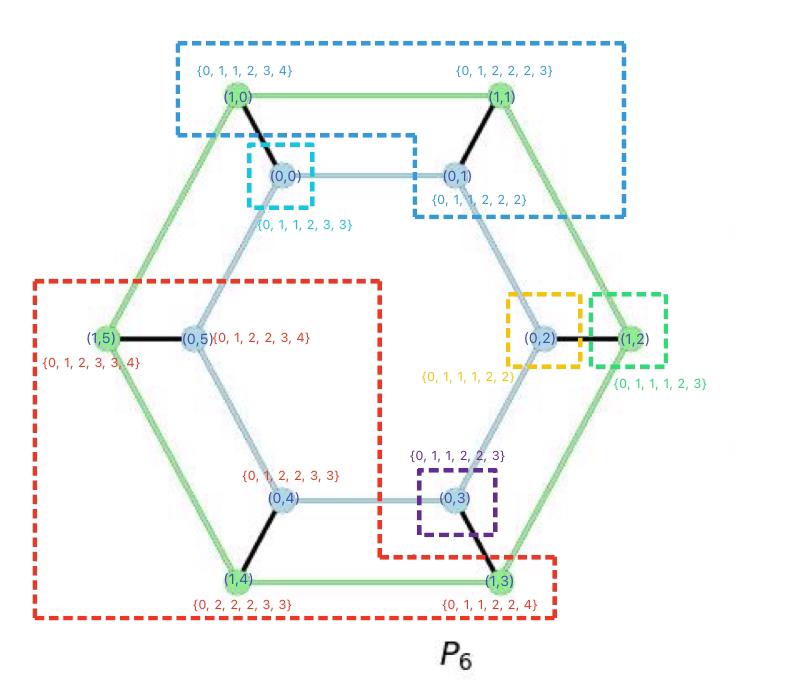}
    \caption{For the prism graph $P_{6}$, the partition 
    $\Pi=\{\{(0,0)\}, \{(0,2)\}, \{(0,3)\},\{(1,2)\}, \{(0,1),(1,0),(1,1)\},\\ \text{\{all remaining vertices\}}\}$ 
    and $\operatorname{mpd}(P_{6}) =6$.}
    \label{fig:Prismwithmpd=6}
\end{figure}
\FloatBarrier

 \begin{figure}[h!]
    \centering
\includegraphics[width=0.5\textwidth]{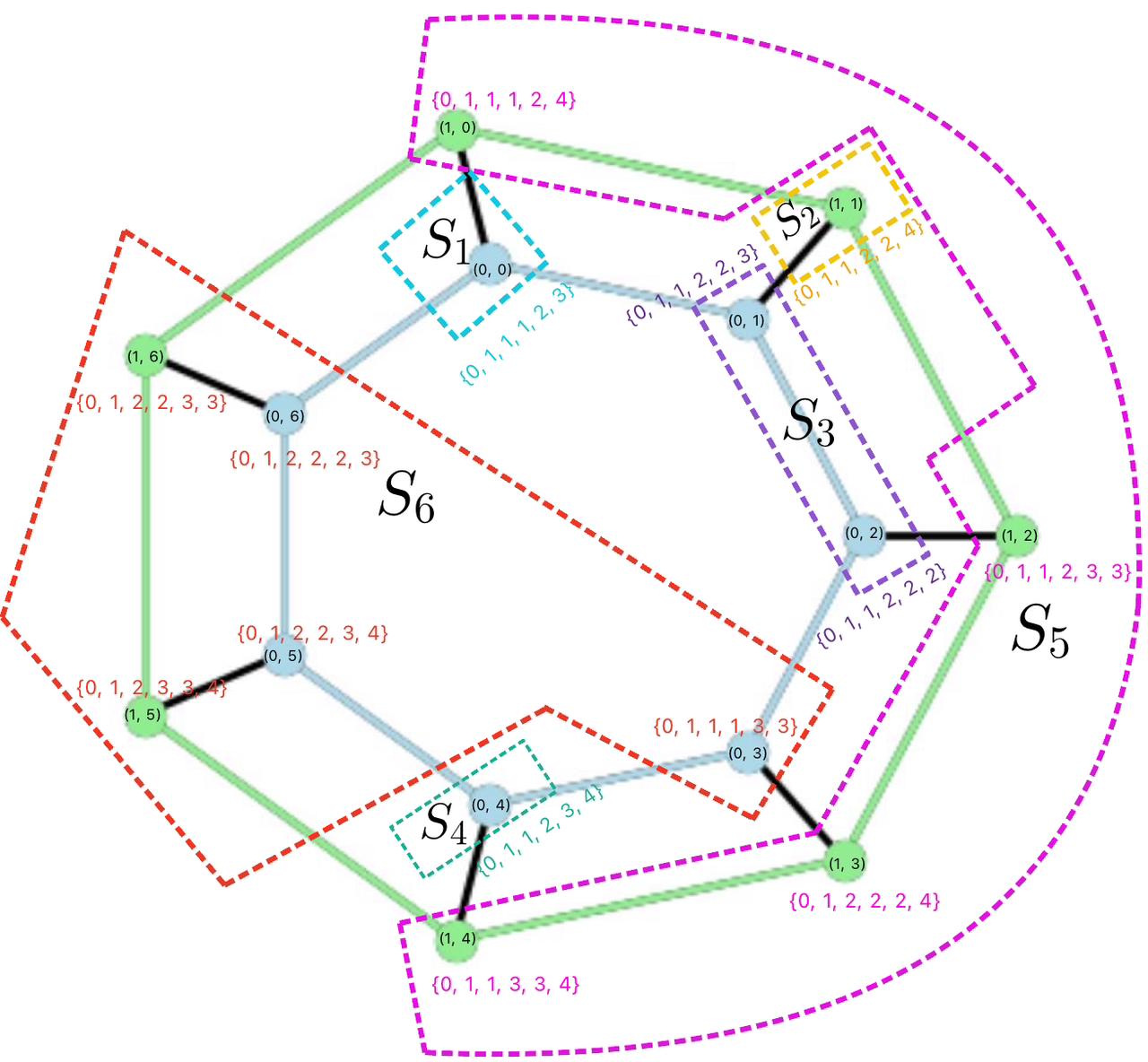}
    \caption{For $P_{7}$, the partition $\Pi=\{S_1, S_2, S_3, S_4, S_5, S_6\}$ is illustrated, showing that $\operatorname{mpd}(P_{7}) = 6$. }
\label{fig:P7}
\end{figure}
\FloatBarrier


\begin{theorem}
Let $P_{m}$ be a prism graph with two $m$-vertex cycles ($m \ge 8$) connected by a layer of vertical edges between the corresponding vertices. Then
\[
\mathrm{mpd}(P_{m}) \leq 6.
\]       
    
\end{theorem}

\begin{proof}

Let $P_{m}$ be a prism graph with $m \ge 8$. We label each vertex by an ordered pair 
\[
v=(i,j), \quad i\in \{0,1\}, \quad \text{and} \quad  \ j\in \{0,...,m-1\}. 
\]

Let 
$\Pi=\{S_1, S_2, S_3, S_4, S_5, S_6\}$
be the partition of $V(P_{m})$ where 
\[
\begin{aligned}
S_1 &= \{(0,0)\}, & S_2= \{(1,1)\}, \\
S_3 &= \{(0,1), (0,2)\}, & S_4= \{(0,4)\}, \\
S_5 &=\{(1,0),(1,2),(1,3),(1,4)\}\\
S_6 &= \{V(P_{m}) \setminus (S_1 \cup S_2 \cup S_3 \cup S_4\cup S_5)\}. \\
\end{aligned}
\]

For any vertex $v=(i,j) \in V(P_{m})$, its multiset partition representation with respect to the above partition is given by 

\[
r_{mp}(v \mid \Pi)=\{d_1, d_2, d_3, d_4, d_5, d_6\},\] where $d_i$ is the distance from $v$ to $S_i$.

\[
\begin{aligned}
d_1 &=i+min \ (|j|, m-|j|), \\
d_2 &=|i-1|+min \ (|j-1|, m-|j-1|), \\
d_3 &=i+ min \ [ \ min \ (|j-1|, m-|j-1|),\ min \ (|j-2|, m-|j-2|)], \\
d_4 &=i+min \ (|j-4|, m-|j-4|), \\
d_5 &=|i-1|+ \ min \ [ \ min \ (|j|, m-|j|), \ min \ (|j-2|, m-|j-2|), \\ &\quad   \ min \ (|j-3|, m-|j-3|),  \ min \ (|j-4|, m-|j-4|)\ ], \\
d_6 &= min_{(x,k) \in S_6} \left (|i-x|+ \ min \ (|j-k|, m-|j-k| \right ).\\
\end{aligned}
\]

For vertices in the parts $S_1, S_2, S_3, S_4, S_5$, the vertices have constant representations and all unique. For vertices in $S_6$, their representations involve larger distances. Moreover, owing to the asymmetric structure of the other five parts, these representations are all pairwise distinct. Hence, $\Pi$ is a resolving multiset partition for $P_{m}$, and $\mathrm{mpd}(P_{m}) \le 6.$
\end{proof}

\begin{figure}[h!]
    \centering
\includegraphics[width=0.4\textwidth]{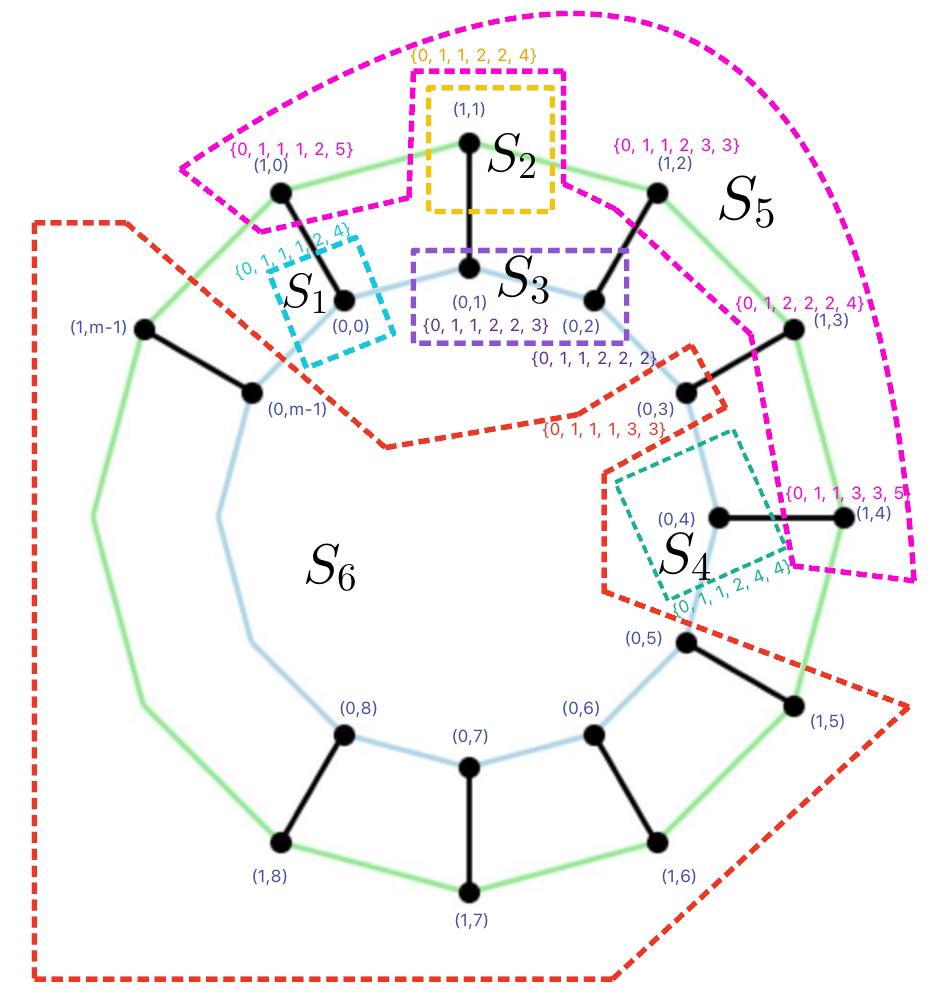}
    \caption{For the prism graph $P_{m}$, with $m \ge 8$, the partition $\Pi=\{S_1, S_2, S_3, S_4, S_5, S_6\}$  is illustrated, showing that $\operatorname{mpd}(P_{m}) \le 6$. The fixed representations of the vertices in $S_1, S_2, S_3, S_4, S_5$ are shown.}
\label{fig:GeneralPrismwithmpd=8}
\end{figure}
\FloatBarrier

From the above results, the differences between the multiset dimension and the multiset partition dimension may be small for some graphs, For example, $mpd(P_n) = {md}(P_n) + 3$, $mpd(C_n) = {md}(C_n) + 1$,  and $mpd(G_{n,m}) = {md}(G_{n,m}) + 1$. While in others the gap can be arbitrarily large such as with Novi's tree, thus, we would like to propose the following problem.

\begin{problem}
How is $\mathrm{mpd}(G)$ bounded in terms of $\mathrm{MD}(G)$ and other related notions of metric-type dimension?
\end{problem}

\section{Graphs with Infinite Multiset Partition Dimension}

It is easy to see the following.

\begin{theorem}
Let $K_n$ be a complete graph. Then
\[
\mathrm{mpd}(K_n) = \infty.
\]          
\end{theorem}

\begin{proof}
Let $K_n$ be the complete graph with n vertices. Suppose $\Pi=\{S_1,S_2,...,S_t\}$ is a resolving multiset partition of $V(K_n)$.
For any vertex $v_i \in V(K_n)$, the multiset partition representation is 
\[r_{mp}(v_i|\Pi)=\{0,1,1,\dots,1\}.\]
consisting of a single $0$ and $t-1$ ones. This is identical for all vertices, thus a contradiction.  
\end{proof}

\begin{theorem}
Let $W_n$ be a wheel graph. Then
\[
\mathrm{mpd}(W_n) = \infty.
\]    
\end{theorem}

\begin{proof}
It is straightforward to see that if there is a part contains more than 2 vertices from the rim, then considering the vertices which only have one neighbor in the part, there are at least two of such kind of vertices and their representation is the same. And it is also clear that if there are two parts each containing a single vertex from the rim, then these two vertices will have the same representation.
\end{proof}

\begin{theorem} \label{wheel}
For the friendship graph $f_n$, the multiset partition dimension is infinite,
\[
\operatorname{mpd}(f_n) = \infty.
\]
\end{theorem}

\begin{proof}
let $f_n$ be the friendship graph with the center $u$ and triangles $\{u, v_{2i-1}, v_{2i}\}$ for $1 \le i \le n$.
Suppose $\Pi=\{S_1,S_2,\ldots, S_t\}$ is a resolving multiset partition of $V(f_n)$.

From Lemma \ref{Lemma2.2partition}, it is clear that
$ v_{2i-1}, v_{2i}$ can not be in the same part.

Let $ v_{2i-1}, v_{2i}$ be in two different partitions $S_p , S_q ,p \neq q$.
$ v_{2i-1} \in s_p$ and $v_{2i} \in s_q$.
\[
d(v_{2i-1},S_p)=0,\qquad d(v_{2i},S_p)=1.
\]
Similarly, 
\[
d(v_{2i},S_q)=0,\qquad d(v_{2i-1},S_q)=1.
\]

For any other partitions $s_j ,j \notin p,q$,
\[
d(v_{2i-1},S_j)=d(v_{2i},S_j).
\]

Therefore, 
\[
r_{mp}(v_{2i-1}\mid\Pi)=r_{mp}(v_{2i}\mid\Pi)
\]
and that contradict $\Pi$ is a resolving multiset partition.

Hence, $mpd(f_n)=\infty$
\end{proof}
\begin{theorem}
For the fan graph $Fan_n$, the multiset partition dimension is infinite, 
\[
\operatorname{mpd}(Fan_n) = \infty.
\]
\end{theorem}

A \textit{$t$-fold wheel graph} $W_{t,n}$ is a graph derived from a wheel by duplicating the hub vertex one or more times, resulting in $t$ hub vertices, each adjacent to all rim vertices, and not adjacent to each other. The \textit{fan graph} $Fan_n$ is the graph obtained from a path $W_n$ by removing an edge on the rim. Along the same line of reasoning as in Theorem \ref{wheel}, it is not hard to see the following.

\begin{theorem}
For the t-fold wheel graph $W_{t,n}$, the multiset partition dimension is infinite, 
\[
\operatorname{mpd}(W_{t,n}) = \infty.
\]
\end{theorem}

\end{document}